\documentclass[12pt,bezier]{article}
\usepackage{times}
\usepackage{booktabs}
\usepackage{pifont}
\usepackage{floatrow}
\usepackage{caption}
\usepackage{mathrsfs}
\usepackage[fleqn]{amsmath}
\usepackage{amsfonts,amsthm,amssymb,mathrsfs,bbding}
\usepackage{txfonts}
\usepackage{graphics,multicol}
\usepackage{graphicx}
\usepackage{color}
\usepackage{mathtools}
\usepackage{amssymb}

\usepackage{bm}
\usepackage{caption}
\evensidemargin=\oddsidemargin\topmargin=-1.5cm

\usepackage{cite}
\usepackage{latexsym,bm}
\usepackage{indentfirst}
\usepackage{color}
\usepackage[colorlinks=true,anchorcolor=blue,filecolor=blue,linkcolor=blue,urlcolor=blue,citecolor=blue]{hyperref}
\usepackage{extarrows}
\usepackage{cite}
\usepackage{latexsym,bm}

\newtheorem{theorem}{Theorem}[section]
\newtheorem{lemma}[theorem]{Lemma}
\newtheorem{problem}[theorem]{Problem}

\theoremstyle{definition}

\addtocounter{section}{0}

\begin{document}
\title{Maxima of the $Q$-index of non-bipartite graphs: forbidden short odd cycles}
\author{{\bf Lu Miao}, ~{\bf Ruifang Liu}\thanks{Corresponding author.
E-mail addresses: miaolu0208@163.com (L. Miao); rfliu@zzu.edu.cn (R. Liu);
jie\_xue@126.com (J. Xue).}, ~{\bf Jie Xue}~\\
{\footnotesize  School of Mathematics and Statistics, Zhengzhou University, Zhengzhou, Henan 450001, China}}

\date{}
\maketitle
{\flushleft\large\bf Abstract}
Let $G$ be a non-bipartite graph which does not contain any odd cycle of length at most $2k+1$.
In this paper, we determine the maximum $Q$-index of $G$ if its order is fixed, and the corresponding extremal graph is uniquely characterized.
Moreover, if the size of $G$ is given, the maximum $Q$-index of $G$ and the unique extremal graph are also proved.

\begin{flushleft}
\end{flushleft}

\section{Introduction}
Let $G$ be a simple graph. Denote by $A(G)$ and $D(G)$ the adjacency matrix and the diagonal degree matrix of $G$, respectively.
The matrix $Q(G)=D(G)+A(G)$ is known as the $Q$-matrix or the signless Laplacian matrix of $G.$
The largest eigenvalue of $Q(G)$, denoted by $q(G)$, is called the $Q$-index or the signless Laplacian spectral radius of $G.$
If $G$ is connected, then cleraly $Q(G)$ is irreducible. By Perron-Frobenius theorem, there
exists a positive unit eigenvector corresponding to $q(G)$,
which is called the Perron vector of $Q(G)$. As usual, we use $n$ and $m$ to denote the order and the size of $G$, respectively.
Recall that $S_{n,k}$ is the join of a clique on $k$ vertices with an independent set of $n-k$ vertices, and $S_{n,k}^{+}$
is obtained by adding an edge within the independent set of $S_{n,k}.$

If a graph does not contain $H$ as a subgraph, then it is said to be \emph{$H$-free}.
Tur\'{a}n-type problem is the problem of determining the largest size of an $H$-free graph of order $n$,
which is a fundamental problem in extremal graph theory. One of the famous results on Tur\'{a}n-type problem is due to Mantel \cite{Mantel} in 1907.

\begin{theorem}[\!\cite{Mantel}]
Every graph of order $n$ with $m>\lfloor \frac{n^{2}}{4}\rfloor$ contains a triangle.
\end{theorem}

The above result is also known as Mantel's theorem.
In other words, Mantel's theorem states that a triangle-free graph has at most $\lfloor\frac{n^{2}}{4}\rfloor$ edges.
Obviously, a bipartite graph is triangle-free, hence it contains at most $\lfloor\frac{n^{2}}{4}\rfloor$ edges.
For non-bipartite graphs, Erd\H{o}s \cite{Bondy2008} improved Mantel's theorem by showing that every non-bipartite triangle-free graph of order $n$ satisfies $m\leq\lfloor\frac{(n-1)^{2}}{4}\rfloor+1.$

In \cite{V5}, Nikiforov formally proposed a spectral version of Tur\'{a}n-type problem as follows.

\begin{problem}[\!\cite{V5}]\label{p1}
What is the maximum spectral radius of an $H$-free graph of order $n$ or size $m$?
\end{problem}

Problem \ref{p1} has been studied in a series of papers (see, for example, \cite{ZLS,LZZ,CFTZ,CS,LP,DKL}).
In 2007, Nikiforov \cite{Nikiforov2007} showed that for $C_4$-free graphs with odd order $n$, the friendship graph is the unique extremal graph with maximum spectral radius. If the order $n$ is even, the extremal graph was determined by Zhai and Wang \cite{ZW}.
For general even cycle, Nikiforov \cite{V5} conjectured that $S_{n,k}^{+}$ is the unique graph with maximum spectral radius among all $C_{2k+2}$-free graphs.
Zhai and Lin \cite{ZL} confirmed the conjecture for $k=2$. Recently, this conjecture was completely solved by Cioab\u{a}, Desai and Tait \cite{CDM}.
Among all non-bipartite $C_{3}$-free graphs with fixed order $n$, the maximum spectral radius was determined by Lin, Ning and Wu \cite{Lin2021}.
For non-bipartite $C_{3}$-free graphs with given size $m$, the maximum spectral radius was improved by Zhai and Shu \cite{Zhai2022}.
As extensions, the maximum spectral radius of $\{C_{3},C_{5},\ldots,C_{2k+1}\}$-free non-bipartite graphs of order $n$ was determined independently
by Lin and Guo \cite{LG2021} and Li, Sun and Yu \cite{LSY}.
Furthermore, Li and Peng \cite{LP2} presented the maximum spectral radius of $\{C_{3},C_{5}\}$-free non-bipartite graphs with given size $m.$
Recently, Lou, Lu and Huang \cite{LLH} determined the maximum spectral radius of $\{C_{3},C_{5},\ldots,C_{2k+1}\}$-free non-bipartite graphs of size $m$
and confirmed a conjecture in \cite{LP2}.

In view of different extremal graphs from the spectral radius in many cases, the $Q$-spectral analogue of Problem \ref{p1} was proposed by Freitas, Nikiforov and Patuzzi \cite{DF2013}.

\begin{problem}[\!\cite{DF2013}]\label{p2}
What is the maximum $Q$-index among all the $H$-free graphs with given order $n$ or size $m$?
\end{problem}

Problem \ref{p2} has been investigated for some special graphs $H$ related to cycles. In \cite{DF2016}, Freitas, Nikiforov and Patuzzi determined the extremal graph with maximum $Q$-index among all $K_{2,s+1}$-free graphs of order $n$. Recently, Zhao, Huang and Guo \cite{Zhao2021} showed that $S_{n,k}$ is the unique graph
attaining the maximum $Q$-index among all $F_{k}$-free graphs of order $n$.
As a generalization of friendship graph $F_{k}$, Chen, Liu and Zhang \cite{Chen20} investigated the maximum $Q$-index of $F_{a_1, \ldots, a_k}$-free graphs of order $n$,
where $F_{a_1, \ldots, a_k}$ is a graph consisting of $k$ cycles of odd length $2a_1+1, \ldots, 2a_k+1$ intersecting in exactly a common vertex.
For $C_{2k+2}$-free graphs of order $n$, Nikiforov and Yuan \cite{Nikiforov2015} proved that $S_{n,k}^{+}$ is the
unique graph with maximum $Q$-index if $n\geq400k^{2}$. Among all $C_{2k+1}$-free graphs of order $n$,
Yuan \cite{Yuan2014} showed that $S_{n,k}$ is the extremal graph with maximum $Q$-index. Very recently,
The maximum $Q$-index of non-bipartite triangle-free graphs with fixed order $n$ or size $m$ was determined in \cite{LMX}, respectively.

Inspired by the above results, we consider the maximum $Q$-index of $\{C_{3},C_{5},\ldots,C_{2k+1}\}$-free non-bipartite graphs
when the order $n$ or the size $m$ is fixed, respectively. Let $H$ be a graph with vertex set $V(H)=\{v_{1},v_{2},\ldots,v_{k}\}$ and $r=(n_{1},n_{2},\ldots,n_{k})$ be a vector of positive integers. A blow-up of $H$, denoted by $H\circ r$, is the graph obtained from $H$ by replacing each vertex $v_{i}$
with an independent set $V_{i}$ of $n_{i}$ vertices and joining each vertex in $V_{i}$ with each vertex in $V_{j}$ for $v_{i}v_{j}\in E(H)$.
The graph $C_{2k+3}\circ(n-2k-2,1,\ldots,1)$ (see Fig. \ref{1}) is a blow-up of $C_{2k+3}$, where $V(C_{2k+3})=\{v_{1},v_{2},\ldots,v_{2k+3}\}$ and $r=(n-2k-2,1,\ldots,1)$. The first main result of this paper presents the maximum $Q$-index of a non-bipartite graph with given order $n$,
if it does not contains any odd cycle of length at most $2k+1$.

\input{f1.TpX}

\begin{theorem}\label{main1}
If $k\geq2$ and $G$ is a $\{C_{3},C_{5},\ldots,C_{2k+1}\}$-free non-bipartite graph of order $n\geq 2k+3$, then
\begin{eqnarray*}
q(G)\leq q(C_{2k+3}\circ(n-2k-2,1,\ldots,1)).
\end{eqnarray*}
Equality holds if and only if $G\cong C_{2k+3}\circ(n-2k-2,1,\ldots,1)$.
\end{theorem}

To identify a vertex $u_{1}$ of $G_{1}$ and a vertex $u_{2}$ of $G_{2}$ is to replace these two vertices by a single vertex
incident to all the edges which were incident to either $u_{1}$ in $G_{1}$ or $u_{2}$ in $G_{2}$. Denote by $C_{2k+3}\bullet K_{1,m-2k-3}$ the graph by identifying a vertex of $C_{2k+3}$ and the central vertex of $K_{1,m-2k-3}$ (see Fig. \ref{2}).
For a non-bipartite graph $G$ with given size $m$, the following result determines the maximum $Q$-index if $G$ contains no any odd cycle of length at most $2k+1.$

\begin{theorem}\label{main2}
Let $k\geq 1$ and $G$ be a $\{C_{3},C_{5},\ldots,C_{2k+1}\}$-free non-bipartite graph of size $m\geq 2k+3$. If $G$ has no isolated vertices,
then
\begin{eqnarray*}
q(G)\leq q(C_{2k+3}\bullet K_{1,m-2k-3}),
\end{eqnarray*}
with equality holds if and only if $G\cong C_{2k+3}\bullet K_{1,m-2k-3}$.
\end{theorem}

The rest of the paper is organized as follows. In Section 2, we introduce the equitable partition with respect to a real square matrix, and give some useful lemmas.
The proofs of Theorems \ref{main1} and \ref{main2} will be presented in subsequent sections.

\section{Some auxiliary results}

Let $G$ be a graph with vertex set $V(G)=\{1,2,\ldots,n\}$.
Let $M$ be a real $n\times n$ matrix defined on the vertices of $G$.
Given a partition $\Pi: V(G)=V_{1}\cup V_{2}\cup\cdots\cup V_{k}$, the matrix $M$ can be partitioned as
$$
M=
\begin{pmatrix}
M_{1,1}&M_{1,2}&\cdots&M_{1,k}\\
M_{2,1}&M_{2,2}&\cdots&M_{2,k}\\
\vdots&\vdots&\ddots&\vdots\\
M_{k,1}&M_{k,2}&\cdots&M_{k,k}
\end{pmatrix}
.$$
The quotient matrix of $M$ with respect to $\Pi$ is defined as the $k\times k$ matrix $B_{\Pi} = (b_{i,j})^{k}_{i,j=1}$, where $b_{i,j}$ is the average value of all row sums of $M_{i,j}$. The partition $\Pi$ is called equitable if each block $M_{i,j}$ of $M$ has constant row sum $b_{i,j}$. Also, we say that the quotient matrix $B_{\Pi}$ is equitable if $\Pi$ is an equitable partition of $M$.

\input{f2.TpX}

The eigenvalues of a given matrix and its quotient matrix satisfy the following property.

\begin{lemma}[\!\cite{Brouwer2011,Godsil2001}]\label{lem2}
 Let $M$ be a real symmetric matrix and $\lambda(M)$ be its largest eigenvalue.
 Let $B_{\Pi}$ be an equitable quotient matrix of $M$. Then the eigenvalues of $B_{\Pi}$ are also eigenvalues of $M$.
 Furthermore, if $M$ is nonnegative and irreducible, then $\lambda(M)=\lambda(B_{\Pi})$.
\end{lemma}

The edge transformation is a classical tool in spectral graph theory. The following lemma presents a transformation increasing the $Q$-index.

\begin{lemma}[\!\cite{Hong2005}]\label{lem3}
Let $G$ be a connected graph and $X=(x_{v})_{v\in V (G)}$ be the Perron vector of $Q(G)$.
Assume that $u_{1}v\notin E(G)$ while $u_{2}v\in E(G)$. If $x_{u_{1}}\geq x_{u_{2}}$, then $q(G-u_{2}v+u_{1}v)>q(G)$.
\end{lemma}

An internal path of $G$ is a path (or cycle) with vertices $v_{1}, v_{2}, \ldots, v_{k}$ (or $v_{1}=v_{k}$) such that $d(v_{1}), d(v_{k})\geq3$ and
$d(v_{2})= \cdots= d(v_{k-1})= 2$, where $k\geq 2$.

\begin{lemma}[\!\cite{FLZ}]\label{lem7}
Let $G$ be a connected graph and $uv$ be an edge on an internal path of $G$. If we subdivide $uv$, that is,
add a new vertex $w$ and substitute $uv$ by a path $uwv$, and denote the new graph by $G_{uv}$, then $q(G_{uv}) < q(G)$.
\end{lemma}

The following two lemmas give upper bounds on the $Q$-index of a graph.

\begin{lemma}[\!\cite{CRS2007}]\label{lem5}
For every graph $G$,
\begin{eqnarray*}
q(G)\leq \max\{d(u)+d(v): uv\in E(G)\}.
\end{eqnarray*}
If $G$ is connected, then equality holds if and only if $G$ is a regular graph or a semi-regular bipartite graph.
\end{lemma}

\begin{lemma}[\!\cite{Merris1998, Feng2009}]\label{lem4}
For every graph $G$,
\begin{eqnarray*}
q(G)\leq \max\{d(u)+m(u): u\in V(G)\},
\end{eqnarray*}
where $m(u)=\frac{1}{d(u)}\sum\limits_{v\in N(u)}d(v)$.
If $G$ is connected, then equality holds if and only if $G$ is either a regular graph or a semi-regular bipartite graph.
\end{lemma}

\section{Proof of Theorem \ref{main1}}

Before proving Theorem \ref{main1}, we first estimate the $Q$-index of the graph $C_{2k+3}\circ(n-2k-2,1,\ldots,1)$.

\begin{lemma}\label{lem6}
Let $k\geq2$ and $n\geq2k+3$. The $Q$-index of the graph $C_{2k+3}\circ(n-2k-2,1,\ldots,1)$ satisfies
\begin{eqnarray*}
q(C_{2k+3}\circ(n-2k-2,1,\ldots,1))>n-2k+1-\frac{3}{2(n-2k+4)}.
\end{eqnarray*}
\end{lemma}

\begin{proof}
Let $G_{0}$ be a graph on $n-2k+4$ vertices, as shown in Fig. \ref{3.0}.
For the vertices of $G_{0}$, we consider a partition
$$\Pi: V(G_{0})=V_{1}\cup V_{2}\cup V_{3}\cup V_{4}.$$
Then the quotient matrix $B_{\Pi}$ of the $Q$-matrix $Q(G_{0})$ with respect to the partition $\Pi$ is as follows:
$$
B_{\Pi}=
\begin{pmatrix}
n-2k-1&n-2k-2&1&0\\
2&2&0&0\\
1&0&2&1\\
0&0&1&1
\end{pmatrix}.
$$
By a direct computation, the characteristic polynomial of $B_{\Pi}$ is
\begin{eqnarray*}
g(x)=|xI_{4}-B_{\Pi}|=x(x^{3}-(n-2k+4)x^{2}+(3n-6k+5)x-(n-2k+4)).
\end{eqnarray*}
Note also that $\Pi$ is an equitable partition.
According to Lemma \ref{lem2}, it follows that $q(G_{0})$ is the largest root of $g(x)=0$.
Since $g(n-2k+1-\frac{3}{2(n-2k+4)})<0$, we have
\begin{eqnarray}\label{eq_a}
q(G_{0})>n-2k+1-\frac{3}{2(n-2k+4)}.
\end{eqnarray}

Since $k\geq2$ and $n\geq2k+3$, one can see that $G_{0}$ is a proper subgraph of $C_{2k+3}\circ(n-2k-2,1,\ldots,1)$.
Hence we have $q(G)>q(G_{0})$. Combining (\ref{eq_a}), the result follows directly.
\end{proof}

\input{f3.TpX}
Next we always assume that $G$ is the extremal graph with maximum $Q$-index among all the non-bipartite $\{C_{3},C_{5},\ldots,$ $C_{2k+1}\}$-free graphs of order $n$.
In the following, we first show that $G$ must be the blow-up of a cycle.

\begin{lemma}\label{lem10}
Suppose that $G$ has the maximum $Q$-index among all $\{C_{3},C_{5},\ldots,$ $C_{2k+1}\}$-free non-bipartite graphs of order $n$. Then we have
\begin{eqnarray*}
G\cong C_{2k+3}\circ(n_{1},n_{2},n_{3},n_{4}, 1,\ldots,1).
\end{eqnarray*}
\end{lemma}

\begin{proof}
Since $G$ is non-bipartite and $\{C_{3},C_{5},\ldots,$ $C_{2k+1}\}$-free, it contains an odd cycle of length at least $2k+3$.
Clearly, $n\geq 2k+3$. We claim that $G$ is connected.
If not, we assume that $G_{1}$ and $G_{2}$ are two connected components of $G$.
By adding an edge between $G_{1}$ and $G_{2}$, we get a new graph $G'.$
Obviously, $G'$ is also a $\{C_{3},C_{5},\ldots,$ $C_{2k+1}\}$-free non-bipartite graph.
However, $q(G')>q(G),$ which contradicts the maximality of $q(G).$

According to Lemma \ref{lem6}, we obtain that
\begin{eqnarray}\label{eq_b}
  q(G)\geq q(C_{2k+3}\circ(n-2k-2,1,\ldots,1))\geq n-2k+1-\frac{3}{2(n-2k+4)}>n-2k.
  \end{eqnarray}
Let $u'v'$ be an edge of $G$ such that
$$d(u')+d(v')=\max\{d(u)+d(v): uv\in E(G)\}.$$
By Lemma \ref{lem5}, we have
\begin{eqnarray}\label{eq1}
q(G) \leq d(u')+d(v').
\end{eqnarray}
Combining (\ref{eq_b}) and (\ref{eq1}), we obtain that
\begin{eqnarray}\label{eq_c}
d(u')+d(v')\geq n-2k+1.
\end{eqnarray}
Since $G$ is $C_{3}$-free, $u'$ and $v'$ cannot have common neighbours.
If $d(u')+d(v')\geq n-2k+2,$ then $|V(G)\backslash\{N(u')\cup N(v')\}|\leq 2k-2$,
and hence the length of any cycle in $G$ is at most $2k+2$.
Moreover, since $G$ is $\{C_{3},C_{5},\ldots, C_{2k+1}\}$-free, it follows that $G$ is bipartite, a contradiction.
Therefore, we obtain that
\begin{eqnarray}\label{eq_d}
d(u')+d(v')\leq n-2k+1.
\end{eqnarray}
It follows from (\ref{eq_c}) and (\ref{eq_d}) that $d(u')+d(v')=n-2k+1.$
This means that $V(G)\backslash\{N(u')\cup N(v')\}$ contains exactly $2k-1$ vertices, say $v_{1},v_{2},\ldots,v_{2k-1}$.
As mentioned above, $G$ has an odd cycle of length at least $2k+3$, this implies that the vertices of $V(G)\backslash\{N(u')\cup N(v')\}$
induce a path $P_{2k-1}=v_{1}v_{2}\cdots v_{2k-1}$ such that $N(v_{1})\cap(N(u')\backslash\{v'\})\neq \emptyset$ and $N(v_{2k-1})\cap(N(v')\backslash\{u'\})\neq \emptyset$.
Moreover, since $G$ is $\{C_{3},C_{5},\ldots,$ $C_{2k+1}\}$-free,
we have $N(v_{i})\cap(N(u')\backslash\{v'\})=\emptyset$ and $N(v_{i})\cap(N(v')\backslash\{u'\})=\emptyset$ for $2\leq i\leq2k-2$.

\input{f4.TpX}\label{f3}

Define $N_{1}=\{v\in N(u') : vv_{1}\in E(G)\}$, $N_{2}=\{v\in N(v'): vv_{2k-1}\in E(G)\}$, $N_{3}=N(u')\backslash (N_{1}\cup \{v'\})$ and $N_{4}=N(v')\backslash (N_{2}\cup \{u'\}).$ According to the above analysis, one can see that $N_{1}\neq \emptyset$ and $N_{2}\neq \emptyset$.
Note that both $N_{1}$ and $N_{4}$ are independent sets as $G$ is $C_{3}$-free.
We claim that $G[N_{1}\cup N_{4}]$ is a complete bipartite graph.
If not, we choose two vertices $u_{1}\in N_{1}$ and $u_{2}\in N_{4}$ such that $u_{1}u_{2}\notin E(G)$.
Clearly, $G+u_{1}u_{2}$ is also a non-bipartite $\{C_{3},C_{5},\ldots,$ $C_{2k+1}\}$-free graph.
However, $q(G+u_{1}u_{2})>q(G),$ a contradiction.
Similarly, one can see that $G[N_{2}\cup N_{3}]$ and $G[N_{3}\cup N_{4}]$ are complete bipartite graphs.
Then at this stage the structure of the $G$ is established in the left part of Fig. \ref{4}.
Let $V_{1}=N_{1}$, $V_{2}=\{u'\}\cup N_{4}$, $V_{3}=N_{3}\cup\{v'\}$ and $V_{4}=N_{2}.$
Then $\cup_{i=1}^{4}V_{i}=V(G)\backslash\{v_{1},v_{2},\ldots,v_{2k-1}\}$ (see Fig. \ref{4}).
Set $|V_{i}|=n_{i}$ for $1\leq i\leq 4.$ Then we have $\sum_{i=1}^{4}n_{i}=n-2k+1$ and $n_{i}\geq1,$ where $1\leq i\leq 4.$
This implies that $G\cong C_{2k+3}\circ(n_{1},n_{2},n_{3},n_{4}, 1,\ldots,1)$, which completes the proof.
\end{proof}

Now we obtain a general structure of the extremal graph $G$, that is, $G\cong C_{2k+3}\circ(n_{1},n_{2},n_{3},n_{4}, 1,\ldots,1)$. Based on it, we prove the following result which is crucial to the proof of Theorem \ref{main1}.

\begin{lemma}\label{lem11}
Let $G\cong C_{2k+3}\circ(n_{1},n_{2},n_{3},n_{4},1,\ldots,1)$ with $\sum_{i=1}^{4}n_{i}=n-2k+1$. Then
there cannot exist exactly two integers $n_{i}$ and $n_{j}$ such that $n_{i}, n_{j}\geq2$ and $|i-j|=1$ in $G$, where $1\leq i,j\leq4$.
\end{lemma}
\begin{proof}
We prove the result by contradiction. Without loss of generality, we may assume that $G\cong C_{2k+3}\circ(n_{1},n_{2},1,\ldots,1)$ with $n_{1}\geq n_{2}\geq2$.
Let $V_{3}=\{u_{3}\}$ and $V_{4}=\{u_{4}\}$. Note that $n_{1}+n_{2}=n-2k-1.$ Then we have $n-2k\geq5$.
Let $G'$ be the graph obtained from $G$ by contracting the internal path $P_{2k-1}=v_{1}v_{2}\cdots v_{2k-1}$ as an edge $v_{1}v_{2k-1}$,
that is, $G'\cong C_{6}\circ(n_{1},n_{2},1,1,1,1)$.
According to Lemma \ref{lem7}, one can see that
\begin{eqnarray}\label{eq g}
q(G)<q(G').
\end{eqnarray}
For the vertices of $G'$, we consider the partition
$$\Pi: V(G')=V_{1}\cup V_{2}\cup\{u_{3}\}\cup\{u_{4}\}\cup\{v_{2k-1}\}\cup\{v_{1}\}.$$
Then the quotient matrix $B_{\Pi}$ of $Q(G')$ is
$$B_{\Pi}=
\begin{pmatrix}
1+n_{2}&n_{2}&0&0&0&1\\
n_{1}&1+n_{1}&1&0&0&0\\
0&n_{2}&1+n_{2}&1&0&0\\
0&0&1&2&1&0\\
0&0&0&1&2&1\\
n_{1}&0&0&0&1&1+n_{1}
\end{pmatrix}
.$$
By calculation, the characteristic polynomial $f(n_{1},n_{2},x)$ of the quotient matrix $B_{\Pi}$ is
\begin{eqnarray*}
f(n_{1},n_{2},x)&=&x^{6}-2(n_{1}+n_{2}+4)x^{5}+\Big((n_{1}+n_{2})^{2}+13(n_{1}+n_{2})+n_{1}n_{2}+23\Big)x^{4}\\
&&-\Big(5(n_{1}+n_{2})^{2}+6n_{1}n_{2}+(n_{1}n_{2}+27)(n_{1}+n_{2})+30\Big)x^{3}\\
&&+\Big(6(n_{1}+n_{2})^{2}+13n_{1}n_{2}+(4n_{1}n_{2}+21)(n_{1}+n_{2})+18\Big)x^{2}\\
&&-\Big((n_{1}+n_{2})^{2}+12n_{1}n_{2}+(3n_{1}n_{2}+5)(n_{1}+n_{2})+4\Big)x.
\end{eqnarray*}
Since the partition $\Pi$ is equitable, it follows from Lemma \ref{lem2} that
$q(G')$ is equal to the largest root of $f(n_{1},n_{2},x)=0.$
Let $G''=C_{6}\circ(n_{1}+1,n_{2}-1,1,1,1,1)$. Then $q(G'')$ is equal to the largest root of $f(n_{1}+1,n_{2}-1,x)=0$.
Note that
\begin{eqnarray*}
&&f(n_{1}+1,n_{2}-1,x)\\
&=&x^{6}-2(n_{1}+n_{2}+4)x^{5}+\Big((n_{1}+n_{2})^{2}+13(n_{1}+n_{2})+(n_{1}+1)(n_{2}-1)+23\Big)x^{4}\\
&&-\Big(5(n_{1}+n_{2})^{2}+6(n_{1}+1)(n_{2}-1)+((n_{1}+1)(n_{2}-1)+27)(n_{1}+n_{2})+30\Big)x^{3}\\
&&+\Big(6(n_{1}+n_{2})^{2}+13(n_{1}+1)(n_{2}-1)+(4(n_{1}+1)(n_{2}-1)+21)(n_{1}+n_{2})+18\Big)x^{2}\\
&&-\Big((n_{1}+n_{2})^{2}+12(n_{1}+1)(n_{2}-1)+(3(n_{1}+1)(n_{2}-1)+5)(n_{1}+n_{2})+4\Big)x.
\end{eqnarray*}
Since $n_{1}+n_{2}=n-2k-1,$ we obtain that
\begin{eqnarray*}
&&f(n_{1},n_{2},x)-f(n_{1}+1,n_{2}-1,x)\\
&=&(2n_{1}+2k+2-n)x(x-3)(x^{2}-(n-2k+2)x+(n-2k+3)).
\end{eqnarray*}
It is easy to see that $2n_{1}+2k+2-n>0$ as $n_{1}\geq n_{2}$ and  $n_{1}+n_{2}=n-2k-1$.
If $x\geq n-2k+1-\frac{3}{2(n-2k+4)}$, then we have $x-3>0$ and $x^{2}-(n-2k+2)x+(n-2k+3)>0$.
This implies that $f(n_{1},n_{2},x)-f(n_{1}+1,n_{2}-1,x)>0$ for $x\geq n-2k+1-\frac{3}{2(n-2k+4)}$.
Recall that $q(G')>q(G)>n-2k+1-\frac{3}{2(n-2k+4)}$. Then we have $f(n_{1},n_{2},x)>f(n_{1}+1,n_{2}-1,x)$ for any $x\geq q(G').$
Hence one can obtain that $q(G'')>q(G')$, that is,
$$q(C_{6}\circ(n_{1}+1,n_{2}-1,1,1,1,1))>q(C_{6}\circ(n_{1},n_{2},1,1,1,1)).$$
Note that $n_{2}\geq 2.$ By using the above operation repeatedly, we have
\begin{eqnarray}\label{eq_e}
q(G')=q(C_{6}\circ(n_{1},n_{2},1,1,1,1))\leq q(C_{6}\circ(n_{1}+n_{2}-2,2,1,1,1,1))
\end{eqnarray}
Since $n_{1}+n_{2}=n-2k-1$, it follows that
$$C_{6}\circ(n_{1}+n_{2}-2,2,1,1,1,1)\cong C_{6}\circ(n-2k-3,2,1,1,1,1).$$
Clearly, $q(C_{6}\circ(n-2k-3,2,1,1,1,1))$ is the largest root of $f(n-2k-3,2,x)=0$.
Note that
\begin{eqnarray*}
  f(n-2k-3,2,x)&=&x^{6}-2(n-2k+3)x^{5}+(n^{2}+13n+4k^{2}-4kn-26k+5)x^{4}\\
  &&-(7n^{2}+21n+28k^{2}-28kn-42k-22)x^{3}\\
  &&+(14n^{2}+3n+56k^{2}-56kn-6k-51)x^{2}\\
  &&-(7n^{2}+3n+28k^{2}-28kn-6k-54)x,
  \end{eqnarray*}
  and
  \begin{eqnarray*}
  f'(n-2k-3,2,x)&=&6x^{5}-10(n-2k+3)x^{4}+4(n^{2}+13n+4k^{2}-4kn-26k+5)x^{3}\\
  &&-3(7n^{2}+21n+28k^{2}-28kn-42k-22)x^{2}\\
  &&+2(14n^{2}+3n+56k^{2}-56kn-6k-51)x\\
  &&-(7n^{2}+3n+28k^{2}-28kn-6k-54).
  \end{eqnarray*}
Note that $n\geq 2k+5.$ It is easy to check that $f(n-2k-3,2,n-2k+1-\frac{3}{2(n-2k+4)})>0$.
Moreover, if $x\geq n-2k+1-\frac{3}{2(n-2k+4)}$, then $f'(n-2k-3,2,x)>0$.
It follows that $f(n-2k-3,2,x)>0$ for $x\geq n-2k+1-\frac{3}{2(n-2k+4)}$, which implies that
\begin{eqnarray}\label{eq_f}
q(C_{6}\circ(n-2k-3,2,1,1,1,1))<n-2k+1-\frac{3}{2(n-2k+4)}.
\end{eqnarray}
Combining (\ref{eq g}), (\ref{eq_e}) and (\ref{eq_f}), it follows that
$$q(G)<q(G')<n-2k+1-\frac{3}{2(n-2k+4)}<q(C_{2k+3}\circ(n-2k-2,1,\ldots,1)),$$
which contradicts the maximality of $q(G)$. This completes the proof.
\end{proof}

Now we are ready to present the proof of Theorem \ref{main1}.

\medskip
\noindent\textbf{Proof of Theorem \ref{main1}.}
Suppose that $G$ is the graph which attains the maximum $Q$-index among all the $\{C_{3},C_{5},\ldots,$ $C_{2k+1}\}$-free non-bipartite graphs of order $n$.
By Lemma \ref{lem10}, we have $G\cong C_{2k+3}\circ(n_{1},n_{2},n_{3},n_{4}, 1,\ldots,1).$
Suppose that the vertices of $G$ are labelled as indicated in the right part of Fig. \ref{4}.
Let $X=(x_{v})_{v\in V (G)}$ be the Perron vector of $Q(G)$.
Assume that $u^{\ast}$ is a vertex of $G$ such that
$$d(u^{\ast})+m(u^{\ast})=\max\{d(u)+m(u):u\in V(G)\}.$$
By Lemma \ref{lem4}, we have $q(G)\leq d(u^{\ast})+m(u^{\ast})$.
According to the vertex $u^{\ast}$, we divide the proof into the following cases.

\vspace{1.5mm}
\noindent{\bf Case 1.} $u^{\ast}\in \{v_{j}:3\leq j\leq2k-3\}$.
\vspace{0.5mm}

Since $d(u^{\ast})=m(u^{\ast})=2$, we have $q(G)\leq d(u^{\ast})+m(u^{\ast})=4$. On the other hand, by Lemma \ref{lem6}, it follows that $q(G)>n-2k+1-\frac{3}{2(n-2k+4)}>n-2k$.
Combining these two inequalities, we obtain that $n-2k<4$, which implies that $n\leq 2k+3$.
Moreover, since $n\geq 2k+3$, it follows that $n=2k+3$, and hence $G\cong C_{2k+3}\circ(1,1,\ldots,1).$

\vspace{1.5mm}
\noindent{\bf Case 2.} $u^{\ast}\in\{v_{2},v_{2k-2}\}$.
\vspace{0.5mm}

Without loss of generality, we may assume that $u^{\ast}=v_{2}$. It follows that
\begin{eqnarray}\label{eq2}
q(G)\leq d(u^{\ast})+m(u^{\ast})=2+\frac{3+n_{1}}{2}=\frac{7+n_{1}}{2}.
\end{eqnarray}
Lemma \ref{lem6} shows that $q(G)>n-2k+1-\frac{3}{2(n-2k+4)}.$ Combining (\ref{eq2}),
we have $$n-2k<q(G)\leq d(u^{\ast})+m(u^{\ast})\leq \frac{7+n_{1}}{2}.$$
Hence $n_{1}>2n-4k-7$.
Recall that $\sum_{i=1}^{4}n_{i}=n-2k+1$ and $n_{i}\geq1$ for $1\leq i\leq4$.
Then we have
\begin{eqnarray*}
3\leq n_{2}+n_{3}+n_{4}=n-2k+1-n_{1}<n-2k+1-(2n-4k-7)=2k+8-n,
\end{eqnarray*}
which implies that $n\leq2k+4$. Since $n\geq2k+3$, it follows that $n=2k+3$ or $n=2k+4$.
If $n=2k+3$, then $G\cong C_{2k+3}\circ(1,1,\ldots,1).$
If $n=2k+4$, then $n_{1}=2$ and $G\cong C_{2k+3}\circ(2,1,\ldots,1)$.

\vspace{1.5mm}
\noindent{\bf Case 3.} $u^{\ast}\in\{v_{1},v_{2k-1}\}$.
\vspace{0.5mm}

Assume, without loss of generality, that $u^{\ast} = v_{1}.$ In this case, we have $d(u^{\ast})=n_{1}+1$ and
$$m(u^{\ast})=\frac{2+n_{1}(1+n_{2})}{n_{1}+1}=n_{2}+1-\frac{n_{2}-1}{n_{1}+1}.$$
It follows that
\begin{eqnarray}\label{eq3.0}
  d(u^{\ast})+m(u^{\ast})=n_{1}+n_{2}+2-\frac{n_{2}-1}{n_{1}+1}=n-2k+3-(n_{3}+n_{4})-\frac{n_{2}-1}{n_{1}+1}.
\end{eqnarray}
By Lemma \ref{lem6} and (\ref{eq3.0}), we have
\begin{eqnarray*}
n-2k<q(G)\leq d(u^{\ast})+m(u^{\ast})=n-2k+3-(n_{3}+n_{4})-\frac{n_{2}-1}{n_{1}+1}.
\end{eqnarray*}
This implies that
\begin{eqnarray*}
2\leq n_{3}+n_{4}<3-\frac{n_{2}-1}{n_{1}+1}\leq3,
\end{eqnarray*}
and hence $n_{3}=n_{4}=1$. Therefore, $G\cong C_{2k+3}\circ(n_{1},n_{2},1,\ldots,1).$
Without loss of generality, we may assume that $n_{1}\geq n_{2}$.
By Lemma \ref{lem11}, we deduce that $n_{2}=1,$ and hence $G\cong C_{2k+3}\circ(n-2k-2,1,1,\ldots,1).$

\vspace{1.5mm}
\noindent{\bf Case 4.} $u^{\ast}\in V_{1}\cup V_{4}.$
\vspace{0.5mm}

By symmetry, we may assume that $u^{\ast}\in V_{1}.$ Clearly, $d(u^{\ast})=n_{2}+1$ and
$$m(u^{\ast})=\frac{1+n_{1}+n_{2}(n_{1}+n_{3})}{n_{2}+1}=n_{1}+\frac{1+n_{2}n_{3}}{n_{2}+1}=n_{1}+n_{3}-\frac{n_{3}-1}{n_{2}+1}.$$
Since $n_{1}+n_{2}+n_{3}+n_{4}=n-2k+1$, we obtain that
\begin{eqnarray}\label{eq3}
d(u^{\ast})+m(u^{\ast})=n-2k+2-n_{4}-\frac{n_{3}-1}{n_{2}+1}.
\end{eqnarray}
By Lemma \ref{lem6} and (\ref{eq3}), we have
\begin{eqnarray*}
n-2k<q(G)\leq d(u^{\ast})+m(u^{\ast})=n-2k+2-n_{4}-\frac{n_{3}-1}{n_{2}+1}.
\end{eqnarray*}
It follows that
\begin{eqnarray*}
1\leq n_{4}<2-\frac{n_{3}-1}{n_{2}+1}\leq2,
\end{eqnarray*}
and hence $n_{4}=1.$ Suppose that $V_{4}=\{u_{4}\}$. Without loss of generality, we assume that $x_{v_{1}}\geq x_{u_{4}}.$
Note that $n_{3}\geq1.$ Assume that $n_{3}\geq 2$. Choose a vertex $z\in V_{3}.$ Let $G'\cong G-zu_{4}+zv_{1}.$
Obviously, $G'$ is a non-bipartite $\{C_{3},C_{5},\ldots,C_{2k+1}\}$-free graph.
However, by Lemma \ref{lem3}, we have $q(G')>q(G),$ which contradicts the maximality of $q(G)$. Hence $n_{3}=1$, that is, $G\cong C_{2k+3}\circ(n_1,n_{2},1,\ldots,1)$.
Without loss of generality, we assume that $n_1\geq n_{2}$. By Lemma \ref{lem11}, we must have $n_{2}=1$. Hence $G\cong C_{2k+3}\circ(n-2k-2,1,\ldots,1)$.

\vspace{1.5mm}
\noindent{\bf Case 5.} $u^{\ast}\in V_{2}\cup V_{3}.$
\vspace{0.5mm}

Suppose, without loss of generality, that $u^{\ast}\in V_{2}$. Then we have $d(u^{\ast})=n_{1}+n_{3}$ and
$$m(u^{\ast})=\frac{n_{1}(1+n_{2})+n_{3}(n_{2}+n_{4})}{n_{1}+n_{3}}=n_{2}+n_{4}-\frac{n_{1}(n_{4}-1)}{n_{1}+n_{3}}.$$
Note that $n_{1}+n_{2}+n_{3}+n_{4}=n-2k+1$. It follows that
\begin{eqnarray}\label{eq4}
d(u^{\ast})+m(u^{\ast})=n-2k+1-\frac{n_{1}(n_{4}-1)}{n_{1}+n_{3}}.
\end{eqnarray}
We claim that $\min\{n_{1},n_{2},n_{4}\}=1$.
Otherwise, it follows from (\ref{eq4}) that
\begin{eqnarray}\label{eq11}
d(u^{\ast})+m(u^{\ast})\leq n-2k+1-\frac{n_{1}}{n_{1}+n_{3}}\leq n-2k+1-\frac{2}{n-2k-3}.
\end{eqnarray}
By Lemma \ref{lem6} and (\ref{eq11}), we have
$$n-2k+1-\frac{3}{2(n-2k+4)}<q(G)\leq n-2k+1-\frac{2}{n-2k-3},$$ a contradiction.
Therefore, $\min\{n_{1},n_{2},n_{4}\}=1$.
In the following, we divide the proof into three subcases.

\vspace{1.5mm}
\noindent{\bf Subcase 5.1.} $n_{1}=1.$
\vspace{0.5mm}

Let $V_{1}=\{u_{1}\}$. Without loss of generality, we may assume that $x_{u_{1}}\geq x_{v_{2k-1}}.$
If $n_{4}\geq 2$, then we choose a vertex $z\in V_{4}.$ Let $G'=G-zv_{2k-1}+zu_{1}.$
Obviously, $G'$ is a non-bipartite $\{C_{3},C_{5},\ldots,C_{2k+1}\}$-free graph.
However, by Lemma \ref{lem3}, we have $q(G')>q(G),$ which contradicts the maximality of $q(G)$.
Hence $n_{4}=1$, that is, $G\cong C_{2k+3}\circ(1,n_{2},n_{3},1,\ldots,1)$.
According to Lemma \ref{lem11}, we have either $n_{2}=1$ or $n_{3}=1$.
This implies that $G\cong C_{2k+3}\circ(n-2k-2,1,\ldots,1)$.

\vspace{1.5mm}
\noindent{\bf Subcase 5.2.} $n_{2}=1.$
\vspace{0.5mm}

In this case, we claim that $n_{1}=1$ or $n_{4}=1$. Otherwise, it follows from (\ref{eq4}) that
\begin{eqnarray*}
d(u^{\ast})+m(u^{\ast})\leq n-2k+1-\frac{n_{1}(n_{4}-1)}{n_{1}+n_{3}}\leq n-2k+1-\frac{2}{n-2k-2}.
\end{eqnarray*}
However, by Lemmas \ref{lem4} and \ref{lem6}, we obtain that
$$q(G)\leq d(u^{\ast})+m(u^{\ast})\leq n-2k+1-\frac{2}{n-2k-2}<q(C_{2k+3}\circ(n-2k-2,1,\ldots,1))$$
which contradicts the maximality of $q(G)$.

If $n_{1}=1$, then $G\cong C_{2k+3}\circ(1,1,n_{3},n_{4},\ldots,1)$. According to Lemma \ref{lem11},
we have either $n_{3}=1$ or $n_{4}=1.$
Hence $G\cong C_{2k+3}\circ(n-2k-2,1,\ldots,1)$.
Suppose that $n_{4}=1$ and $V_{4}=\{u_{4}\}$. Without loss of generality, we assume that $x_{v_{1}}\geq x_{u_{4}}.$
If $n_{3}\geq 2$, then we choose a vertex $z\in V_{3}.$ Let $G'=G-zu_{4}+zv_{1}.$
Obviously, $G'$ is a non-bipartite $\{C_{3},C_{5},\ldots,C_{2k+1}\}$-free graph.
However, by Lemma \ref{lem3}, we have $q(G')>q(G),$ a contradiction. Therefore, $n_{3}=1$.
This implies that $G\cong C_{2k+3}\circ(n-2k-2,1,\ldots,1)$.

\vspace{1.5mm}
\noindent{\bf Subcase 5.3.} $n_{4}=1.$
\vspace{0.5mm}

Let $V_{4}=\{u_{4}\}$. Now we may assume that $x_{v_{1}}\geq x_{u_{4}}.$
If $n_{3}\geq 2$, then we choose a vertex $z\in V_{3}.$ Let $G'=G-zu_{4}+zv_{1}.$
Obviously, $G'$ is a non-bipartite $\{C_{3},C_{5},\ldots,C_{2k+1}\}$-free graph.
However, by Lemma \ref{lem3}, we have $q(G')>q(G),$ which contradicts the maximality of $q(G)$.
Hence $n_{3}=1$, that is, $G\cong C_{2k+3}\circ(n_{1},n_{2},1,1,\ldots,1)$.
According to Lemma \ref{lem11}, we have either $n_{1}=1$ or $n_{2}=1$.
This implies that $G\cong C_{2k+3}\circ(n-2k-2,1,\ldots,1)$.
\hspace*{\fill}$\Box$

\section{Proof of Theorem \ref{main2}}

Now, we are in a position to give the proof of Theorem \ref{main2}.

\medskip
\noindent  \textbf{Proof of Theorem \ref{main2}}.
Suppose that $G$ attains the maximum $Q$-index among all non-bipartite $\{C_{3},C_{5},\ldots,C_{2k+1}\}$-free graphs of size $m.$
Note that a non-bipartite graph $G$ of size $m\leq2k+3$ must contain an odd cycle of length $2i+1$ unless $G\cong C_{2k+3}$, where $1\leq i\leq k$.
In the following, it suffices to consider the case $m\geq2k+4.$
Our goal is to show that $G\cong C_{2k+3}\bullet K_{1,m-2k-3}$.
First we claim that $G$ is connected.
Otherwise, suppose that $G$ contains $k$ connected components $G_{1},G_{2},\ldots,G_{k}$,
where $q(G)=q(G_{i_{0}})$ for some $i_{0}\in \{1,2,\ldots,k\}$.
Select a vertex $u_{i}\in V(G_{i})$ for each $i\in \{1,2,\ldots,k\},$ and let $G'$ be the graph obtained from
$G$ by identifying vertices $u_{1},u_{2},\ldots,u_{k}.$
It is easy to see that $G'$ is also a non-bipartite $\{C_{3},C_{5},\ldots,C_{2k+1}\}$-free graph of size $m$.
Moreover, $G'$ contains $G_{i_{0}}$ as a proper subgraph, and hence $q(G')>q(G_{i_{0}})=q(G)$, contradicting the maximality of $q(G).$

Note that $K_{1,m-2k-1}$ is a proper subgraph of $C_{2k+3}\bullet K_{1,m-2k-3}$.
Then we have $$q(C_{2k+3}\bullet K_{1,m-2k-3})>q(K_{1,m-2k-1})=m-2k.$$
Combining the fact $q(G)\geq q(C_{2k+3}\bullet K_{1,m-2k-3})$, it follows that
\begin{eqnarray}\label{eq10}
q(G)>m-2k.
\end{eqnarray}

Suppose that $m\leq4k+3$. Let $C$ be the shortest odd cycle of $G$, then $|C|\geq2k+3$.
If $|C|\geq2k+5$, then $d(u)+d(v)\leq m-2k-1$ for each edge $uv\in E(G)$. By Lemma \ref{lem5},
it follows that $q(G)\leq m-2k-1$, contradicting (\ref{eq10}).
Therefore, we obtain that $|C|=2k+3$.
If $G\ncong C_{2k+3}\bullet K_{1,m-2k-3}$, then either $G\cong G_{1}$ (see Fig. \ref{5}) or $d(u)+d(v)\leq m-2k$ for each edge $uv\in E(G)$.
For the latter case, by Lemma \ref{lem5}, we obtain that $q(G)\leq m-2k$, contradicting (\ref{eq10}).
Suppose that $G\cong G_{1}$. Let $X=(x_{v})_{v\in V (G)}$ be the Perron vector of $Q(G)$.
We may assume that $x_{u_{1}}\geq x_{u_{2}}$. Let $N_{j}$ be the set of neighbours of $u_{j}$ in $V(G)\backslash V(C)$, where $j=1,2$ (see Fig. \ref{5}). According to Lemma \ref{lem3}, we have $q(G-u_{2}v+u_{1}v)>q(G)$ for any $v\in N_{2}$, a contradiction.
This implies that $G\cong C_{2k+3}\bullet K_{1,m-2k-3}$.
\input{f5.TpX}

Assume now that $m\geq 4k+4$.
Let $\Delta$ be the maximum degree of $G.$ If $\Delta\geq m-2k,$ then either $G$ is bipartite or $G$ contains an odd cycle length of $2i+1$
for $1\leq i\leq k$, a contradiction.
Hence we have $\Delta \leq m-2k-1.$
Let $u^{\ast}$ be a vertex of $G$ such that
$$d(u^{\ast})+m(u^{\ast})=\max\{d(u)+m(u): u\in V(G)\}.$$
By Lemma \ref{lem4} and (\ref{eq10}), we have
\begin{eqnarray}\label{eq7}
  d(u^{\ast})+m(u^{\ast})=d(u^{\ast})+\frac{\sum_{v\in N(u^{\ast})}d(v)}{d(u^{\ast})}\geq q(G)>m-2k.
\end{eqnarray}
Now we divide the proof into three cases according to different values of $d(u^{\ast}).$

\vspace{1.5mm}
\noindent{\bf Case 1.} $d(u^{\ast})=1$.
\vspace{0.5mm}

Let $v$ be the only neighbour of $u^{\ast}.$ It follows from (\ref{eq7}) that $1+d(v)>m-2k$, and hence $d(v)>m-2k-1,$ which contradicts the fact $\Delta\leq m-2k-1$.

\vspace{1.5mm}
\noindent{\bf Case 2.} $2\leq d(u^{\ast})\leq m-2k-2$.
\vspace{0.5mm}

Let $N(u^{\ast})$ be the set of neighbours of $u^{\ast}$ in $G$. Since $G$ is $C_{3}$-free, $N(u^{\ast})$ is an independent set.
Hence we have
\begin{eqnarray*}
\sum_{v\in N(u^{\ast})}d(v)=|E\big(N(u^{\ast}), V(G)\backslash N(u^{\ast})\big)|\leq m,
\end{eqnarray*}
where $E\big(N(u^{\ast}), V(G)\backslash N(u^{\ast})\big)$ denotes the set of edges between $N(u^{\ast})$ and $V(G)\backslash N(u^{\ast})$.
It follows that
\begin{eqnarray*}
d(u^{\ast})+m(u^{\ast})=d(u^{\ast})+\frac{\sum_{v\in N(u^{\ast})}d(v)}{d(u^{\ast})}\leq d(u^{\ast})+\frac{m}{d(u^{\ast})}.
\end{eqnarray*}
Note that the function $f(x)=x+\frac{m}{x}$ is convex for $x>0.$
Hence the maximum value of $d(u^{\ast})+\frac{m}{d(u^{\ast})}$ is attained at $d_{u^{\ast}}=2$ or $d_{u^{\ast}}=m-2k-2$.
Since $m\geq4k+4$, we obtain that
\begin{eqnarray*}
d(u^{\ast})+\frac{m}{d(u^{\ast})}\leq\max\{2+\frac{m}{2}, m-2k-2+\frac{m}{m-2k-2}\}\leq m-2k,
\end{eqnarray*}
which contradicts (\ref{eq7}).

\vspace{1.5mm}
\noindent{\bf Case 3.} $d(u^{\ast})=m-2k-1$.
\vspace{0.5mm}

Let $T=V(G)\backslash (N(u^{\ast})\cup\{u^{\ast}\})$, and let $e(T)$ denote the number of edges in the induced subgraph $G[T]$.
Since $G$ is $C_{3}$-free, $N(u^{\ast})$ is an independent set.
Note that $G$ is a connected non-bipartite graph. Therefore, $e(N(u^{\ast}), T)\geq1$ and $e(T)\geq1$.
Since $G$ is $\{C_{3},C_{5},\ldots,C_{2k+1}\}$-free and non-bipartite, it follows that $e(T)=2k-1$ and $G[T]$ is a path $P_{2k}=v_{1}v_{2}\cdots v_{2k}$.
Moreover, one can see that there are exactly two edges between $N(u^{\ast})$ and $T$
such that these two edges are independent and incident to $v_{1}$ and $v_{2k}$, respectively.
This implies that $G\cong C_{2k+3}\bullet K_{1,m-2k-3}$. The proof is completed.\hspace*{\fill}$\Box$
\medskip

\section*{Acknowledgements}

The research of Ruifang Liu is supported by National Natural Science Foundation of China (Nos. 11971445 and 12171440) and Natural Science Foundation of Henan (No. 202300410377). The research of Jie Xue is supported by National Natural Science Foundation of China (No. 12001498)
and China Postdoctoral Science Foundation (No. 2022TQ0303).

\end{document}